\documentclass[12pt,a4paper,oneside]{amsart}
\usepackage[left=1.1in, right=0.8in, top=1.0in, bottom=1.0in]{geometry}
\usepackage{graphicx}
\usepackage{setspace}
\usepackage{indentfirst} 
\usepackage{cases}
\usepackage{wrapfig}
\usepackage[toc,page]{appendix}
\usepackage{amsmath}
\usepackage{amsthm}
\usepackage{amssymb}
\usepackage{enumerate}
\usepackage{mathrsfs}
\usepackage{dcolumn}
\newcolumntype{L}{D{.}{.}{2,5}}
\theoremstyle{plain}
\newtheorem{thm}{Theorem}

\newtheorem{proposition}{Proposition}
\newtheorem{lemma}{Lemma}
\newtheorem{corollary}[proposition]{Corollary}
\newtheorem{mydef}{Definition}

\DeclareMathOperator{\Id}{Id}
\DeclareMathOperator{\Fix}{Fix}

\DeclareMathOperator{\co}{co}

\DeclareMathOperator{\CAT}{CAT}
\newtheorem{remark}{Remark}
\theoremstyle{definition}

\usepackage{cite}
 \usepackage{hyperref}
 \hypersetup{
    colorlinks=true,
    linkcolor=black,
    filecolor=magenta,      
    urlcolor=cyan}

\begin{document}

\title{On a notion of averaged operators in $\CAT(0)$ spaces}
\author{
{Arian B\"erd\"ellima}}

 \address{University of G\"ottingen}\address{37083 G\"ottingen, Germany}
  \thanks{This work was supported by Deutscher Akademischer Austauschdienst (DAAD).\\
  Original results from author's Thesis.
  Electronic address: \texttt{berdellima@gmail.com}}

 \begin{abstract}
Averaged operators have played an important role in fixed point theory in Hilbert spaces. They emerged as a necessity to obtain solutions to fixed point problems where the underlying operator is not contractive and thus renders Banach fixed point theorem inaccessible. We introduce a notion of {\em averaged operator} in the broader class of $\CAT(0)$ spaces. We call these operators $\alpha$-firmly nonexpansive and develop basic calculus rules for the quasi $\alpha$-firmly nonexpansive operators. In particular compositions of quasi $\alpha$-firmly nonexpansive operators is quasi $\alpha$-firmly nonexpansive and {\em convex combination} of a finite family of quasi $\alpha$-firmly nonexpansive operators is again quasi $\alpha$-firmly nonexpansive. For a nonexpansive operator $T:X\to X$ acting on a $\CAT(0)$ space $X$ we show that the iterates $x_n:=Tx_{n-1}$ converge weakly to some element in the fixed point set $\Fix T$ whenever $T$ is quasi $\alpha$-firmly nonexpansive. Moreover under a certain regularity condition the projections $P_{\Fix T}x_n$ converge strongly to this weak limit. 
Our theory is illustrated with two classical examples of cyclic and averaged projections. 
\end{abstract} 

\maketitle

\section{Introduction}\label{s:intro}
In a Hilbert space $(\mathcal{H},\|\cdot\|)$ an operator $T:\mathcal{H}\to \mathcal{H}$ is 
$\alpha$-averaged
with averaging constant $\alpha\in(0,1)$ if there exists a nonexpansive operator $A:\mathcal{H}\to\mathcal{H}$ such that 
\begin{equation}
 \label{eq:avg-operator}
 Tx=(1-\alpha)x+\alpha Ax,\;\forall x\in \mathcal{H}.
\end{equation}
This class of operators originates from the classical Krasnoselskij--Mann iteration method  attributed to Krasnoselskij \cite{Krasnoselski} and Mann \cite{Mann} which was successfully employed in a series of seminal paper by Browder \cite{Browder}, \cite{Browder1}, \cite{Browder2}, \cite{Browder3} and Browder and Petryshin \cite{BrowderPetryshin} about fixed points of nonlinear mappings in Banach spaces and their relations to \textit{variational inequalities}. In the particular case when $\alpha=1/2$ the class of operators $T$ in \eqref{eq:avg-operator} coincides with the set of {\em firmly nonexpansive} operators and in general the class of $\alpha$-averaged operators coincides with the set of {\em $\alpha$-firmly nonexpansive} operators (see Bauschke and Combette \cite{Bauschkebook}). While averaged operators are well understood in the setting of a Hilbert space, there is less concesus as to what could be a proper definition for such an operator in the setting of a more general metric space that might even lack a linear structure. Our attention in this note will be on the broader class of $\CAT(0)$ spaces, which includes in particular Hilbert spaces. While in general a $\CAT(0)$ space $(X,d)$, where $d$ is its canonical metric, is a nonlinear metric space, its geodesic space structure indeed allows for a meaningful interpretation of the expression in \eqref{eq:avg-operator}. In the context of a $\CAT(0)$ space $Tx$ is the point on the geodesic segment connecting $x$ with $Ax$ that is $\alpha d(x,Ax)$ units away from $x$. However such a notion comes with two main shortcomings. First note that $Tx$ being on the geodesic segment $[x,Ax]$ forces the existence of an operator $A$ such that the geodesic segment $[x,Tx]$ connecting $x$ with $Tx$ be extended beyond point $Tx$. This requires $(X,d)$ to have the so called {\em geodesic extension property}. While a large class of $\CAT(0)$ spaces satisfy this condition, not all of them do (see Bridson and Haefliger \cite{Brid}). Second, even if a $\CAT(0)$ space has the geodesic extension property, this extension is not unique in general like for instance in polyhedral complexes. This bifurcation of geodesics introduces further ambiguity as for which extension one should select to define the operator $A$. Nevertheless there is a widely accepted notion for a firmly nonexpansive operator in a $\CAT(0)$ space, see Ba\v cak \cite{Bacak} or Lopez et al. \cite{Lopez}, that was adopted from the original definition of Bruck \cite{Bruck} who used it for operators acting in a Banach space. Browder \cite{Browder} used the same concept as Bruck, but under the name firmly contractive, in the setting of a Hilbert space. The original definition of Bruck does in fact coincide with the current notion of firmly nonexpansiveness in a Hilbert space. While this concept has proved to be very useful in fixed point theory in a $\CAT(0)$ space, it is yet not clear how to extend it to a definition which allows for the general averagedness of an operator. Motivated by this lack of a concept for an averaged operator in a $\CAT(0)$ space and its potential applications to fixed point theory the author considers in his Thesis \cite{Berdellima} a notion of firmly nonexpansiveness which avoids the ambiguities of geodesic extension and geodesic bifurcation and moreover it allows for an appropriate definition that essentially captures the idea of averagedness for an operator. The whole construct is based on a real valued mapping $\Delta_T$ associated to every operator $T$ 
which the author refers to as the {\em discrepancy} of an operator. Using $\Delta_T$ we are able to provide a meaningful and unambiguous notion for a $\alpha$-firmly nonexpansive operator in a $\CAT(0)$ space.

This short note is organized as follows. In Section \ref{s:concepts} we introduce the necessary concepts and notations. Later in Section \ref{s:calculus} we establish some basic calculus rules for the class of {\em quasi $\alpha$-firmly nonexpansive} operators. In Section \ref{s:projection} we present a general convergence theorem in a complete $\CAT(0)$ space illustrated with the classical cyclic and averaged projections.

\section{Preliminary definitions and notations}\label{s:concepts}
\subsection{$\CAT(0)$ spaces} Let $(X,d)$ be a metric space. A {\em geodesic segment} starting from $x\in X$ and ending at $y\in X$ is a mapping $\gamma:[0,1]\to X$ such that $\gamma(0)=x,\gamma(1)=y$ and $d(\gamma(t_1),\gamma(t_2))=|t_1-t_2|d(x,y)$ for all $t_1,t_2\in [0,1]$. Often we will denote a geodesic segment connecting $x$ and $y$ by $[x,y]$.
A metric space where every pair of elements is connected by a geodesic is called a {\em geodesic metric space}. If this geodesic is unique then we say the metric space is {\em uniquely geodesic}. In a uniquely geodesic space given $x,y\in X$ and $t\in[0,1]$ we denote by $x_t:=(1-t)x\oplus ty$ the element on the geodesic segment $[x,y]$ such that $d(x_t,x)=td(x,y)$. Given three elements $x,y,z\in X$ a {\em geodesic triangle}, denoted by $\Delta(x,y,z)$, is the union of the geodesic segments $[x,y],[y,z]$ and $[z,x]$. Note that unless $X$ is uniquely geodesic the geodesic triangle is not unique. To each geodesic triangle corresponds a {\em comparison triangle} in $\mathbb{R}^2$. A comparison triangle is the union of three line segments in $\mathbb{R}^2$ determined by three points $\overline{x},\overline{y},\overline{z}\in\mathbb{R}^2$ such that $d(x,y)=\|\overline{x}-\overline{y}\|, d(y,z)= \|\overline{y}-\overline{z}\|$ and $d(z,x)=\|\overline{z}-\overline{x}\|$. We denote it by $\Delta(\overline{x},\overline{y},\overline{z})$. In a uniquely geodesic metric space the comparison triangle is unique up to an isometry. Given a point $p\in \Delta(x,y,z)$, say $p\in[x,y]$, a point $\overline{p}\in\Delta(\overline{x},\overline{y},\overline{z})$ is a {\em comparison point} for $p$ if $\overline{p}\in[\overline{x},\overline{y}]$ and $d(x,p)=\|\overline{x}-\overline{p}\|$. A geodesic triangle satisfies the {\em $\CAT(0)$ inequality} if for any two points $p,q$ in the geodesic triangle we have
\begin{equation}
 \label{eq:CAT(0)-ineq}
 d(p,q)\leqslant \|\overline{p}-\overline{q}\|.
\end{equation}
If this inequality holds for any geodesic triangle then $(X,d)$ is said to be a $\CAT(0)$ space. Often one refers to a $\CAT(0)$ space as a metric space of nonpositive curvature in the sense of Alexandrov (see Ballman \cite{Ballman} or Bridson and Haefliger \cite{Brid}). This definition using comparison triangles was originally introduced by Alexandrov \cite{Alexandrov}. Mikhael Gromov gave prominence to Alexandrov’s
definition for what it might mean for a metric space to have a
curvature bounded above by a real number $\kappa$ which he called the
$\CAT(\kappa)$ inequality (see Gromov \cite{Gromov},\cite{Gromov2}).
An equivalent characterization of a $\CAT(0)$ space \cite[Definition 1.2.1]{Bacak} is given by the following quadratic inequality
\begin{equation}
 \label{eq:quadratic}
 d(x_t,z)^2\leqslant (1-t)d(x,z)^2+td(y,z)^2-t(1-t)d(x,y)^2,\;\forall x,y,z\in X,\forall t\in[0,1].
\end{equation}
It follows from \eqref{eq:quadratic} that a $\CAT(0)$ space is uniquely geodesic and moreover the geodesics depend continuously on their endpoints \cite[Lemma 1.2.2]{Bacak}. 

\subsection{Convexity and projections}
Let $(X,d)$ be a $\CAT(0)$ space. A set $C\subseteq X$ is convex if for any $x,y\in C$ the geodesic segment $[x,y]$ is entirely contained in $C$. For a given set $S\subseteq X$ the convex hull of $S$, denoted by $\co S$, is the smallest convex set in $X$ that contains $S$. Clearly $S=\co S$ if and only if $S$ is a convex set. The notion of convexity is extended naturally to real valued functions as follows. 
A function $f:X\to \mathbb{R}$ is convex if for any $x,y\in X$ and $t\in[0,1]$ 
 \begin{equation}
  \label{eq:convexfunction}
  f(x_t)\leqslant (1-t)f(x)+tf(y).
 \end{equation}
A function $f:X\times X\to \mathbb{R}$ is jointly convex if for all $x_0,x_1,y_0,y_1\in X$ and $t\in[0,1]$ 
\begin{equation}
 \label{eq:jointlyconvexfunciton}
 f(x_t,y_t)\leqslant (1-t)f(x_0,y_0)+tf(x_1,y_1)
\end{equation}
where $x_t:=(1-t)x_0\oplus tx_1$ and $y_t:=(1-t)y_0\oplus ty_1$. The metric function $f(x,y):=d(x,y)$ is jointly convex and in particular convex. A function $f:X\to \mathbb{R}$ is {\em strongly convex} if there exists some parameter $\mu>0$ such that 
\begin{equation}
 \label{eq:stronglyconvex}
 f(x_t)\leqslant (1-t)f(x)+tf(y)-\frac{\mu}{2}t(1-t)d(x,y)^2,\;\forall x,y\in X,\forall t\in[0,1].
\end{equation}
From \eqref{eq:quadratic} it immediately follows that $f(x):=d(x,z)^2$ is strongly convex with parameter $\mu=2$ for every $z\in X$. In fact this characterizes all metric spaces that are $\CAT(0)$. An important result concerning lower semicontinuous strongly convex functions is the following. Note that a function $f:X\to X$ is lower semicontinuous whenever $f(x)\leqslant \liminf_{y\to x}f(y)$ for all $x\in X$.
\begin{proposition}[\hspace{-0.03em}{\cite[Proposition 2.2.17]{Bacak}}]
\label{prop1}
 Let $(X,d)$ be a complete $\CAT(0)$ space. Let $f: X\to (-\infty,+\infty]$ be a lower semicontinuous strongly convex function with parameter $\mu>0$. Then $f$ has a unique minimizer $x\in X$ and each minimizing sequence converges to $x$. Moreover 
\begin{equation}
 \label{eq:strcvxidentity}
 f(x)+\frac{\mu}{2} d(x,y)^2\leqslant f(y),\;\forall y\in X
\end{equation}
\end{proposition}
Let $C\subseteq X$. The metric projection operator $P_C:X\to C$ on the set $C$ is defined as $P_Cx:=\{y\in C\;:\; d(x,y)=d(x,C)\}$. When $C$ is a closed convex set in a complete $\CAT(0)$ space \footnote{Complete $\CAT(0)$ spaces are also known as {\em Hadamard spaces}.} then $P_Cx$ is nonempty and unique for every $x\in X$ (\cite[Proposition 2.4]{Berdellima}). Moreover the following inequality holds
\begin{equation}
 \label{eq:projectionsineq}
 d(x,P_Cx)^2+d(P_Cx,y)^2\leqslant d(x,y)^2,\;\forall x\in X,\forall y\in C.
\end{equation}

\subsection{Nonexpansive operators} Let $(X,d)$ be a metric space. An operator $T:X\to X$ is nonexpansive on $X$ if 
\begin{equation}
 \label{eq:nonexpansive}
 d(Tx,Ty)\leqslant d(x,y),\;\forall x,y\in X.
\end{equation}
Let $\Fix T:=\{x\in X\,:\,Tx=x\}$ denote the fixed point set of an operator $T$.
If $\Fix T\neq\emptyset$ and inequality \eqref{eq:nonexpansive} holds for all $x\in X $ and for all $y\in \Fix T$ we say that $T$ is {\em quasi nonexpansive}.
In a $\CAT(0)$ space the class of (quasi) nonexpansive operators is well behaved under compositions and convex combinations. More precisely if $T_1,T_2:X\to X$ are two (quasi) nonexpansive operators then $T_1T_2$ (also $T_2T_1$) is (quasi) nonexpansive. In general the composition of any finite number of (quasi) nonexpansive operators is also (quasi) nonexpansive. Likewise for any $s\in [0,1]$ the {\em convex combination} of two operators $T:=(1-s)T_1\oplus sT_2$ is (quasi) nonexpansive whenever $T_1,T_2$ are (quasi) nonexpansive. For more details and calculations refer to \cite[Chapter 7]{Berdellima}. 

\subsection{Discrepancy of an operator} Let $(X,d)$ be a metric space. Following Berg and Nikolaev \cite{Berg} each pair $(x,y)\in X\times X$ determines a {\em bound vector} denoted by $\overrightarrow{xy}$. The point $x$ is called the tail of $\overrightarrow{xy}$ and $y$ is called the head. The length of a bound vector $\overrightarrow{xy}$ equals the metric distance $d(x,y)$. Let 
 \begin{equation}
 \label{eq:qlin}
 \langle\overrightarrow{xz},\overrightarrow{yw}\rangle:=\frac{1}{2}\Big(d(x,w)^2+d(z,y)^2-d(x,y)^2-d(z,w)^2\Big). 
 \end{equation}
In a Hilbert space $\mathcal{H}$ expression in \eqref{eq:qlin} coincides with the canonical inner product, i.e. $\langle\overrightarrow{xz},\overrightarrow{yw}\rangle\equiv \langle z-x, w-y\rangle$ for all $x,y,z,w\in\mathcal{H}$. 
Given an operator $T:X\to X$ the {\em discrepancy} of $T$ is the mapping $\Delta_T:X\times X\to\mathbb{R}$ defined as 
\begin{equation}
 \label{eq:discrepancy}
 \Delta_T(x,y):=\langle\overrightarrow{xy},\overrightarrow{TxTy}\rangle
,\;\forall x,y\in X.
\end{equation}
The discrepancy is symmetric i.e. $\Delta_T(x,y)=\Delta_T(y,x)$ for all $x,y\in X$, it vanishes identically whenever $T$ is a constant mapping and it equals $d(x,y)^2$ if $T$ coincides with the identity operator $\Id$. Moreover $\Delta_T$ is continuous on $X\times X$ if $T$ is continuous on $X$.

\subsection{$\alpha$-firmly nonexpansive operators}
Let $(\mathcal{H},\|\cdot\|)$ be a Hilbert space. An operator $T:\mathcal{H}\to\mathcal{H}$ is $\alpha$-firmly nonexpansive on $\mathcal{H}$ if there is a constant $\alpha\in(0,1)$ such that 
\begin{equation}
 \label{eq:avg-identity}
 \|Tx-Ty\|^2+(1-2\alpha)\|x-y\|^2\leqslant 2(1-\alpha)\langle x-y,Tx-Ty\rangle,
\qquad \forall x,y\in \mathcal{H}.
\end{equation}
In a Hilbert space it is known that an operator $T$ is $\alpha$-firmly nonexpansive if and only if $T$ is an $\alpha$-averaged operator with averaging constant $\alpha\in(0,1)$ (see \cite[Proposition 2.1]{BauCom11} or \cite[Proposition 2.1]{LukTamTha18} for the pointwise (almost) $\alpha$-averaged variant). 
Using the discrepancy of an operator we can extend such a notion to a $\CAT(0)$ space in its generality. 
\begin{mydef}
 Let $(X,d)$ be a $\CAT(0)$ space and $D,E\subseteq X$ nonempty sets. An operator $T:X\to X$ 
 is $\alpha$-firmly nonexpansive on $D\times E$ if there exists an $\alpha\in(0,1)$ such that 
\begin{equation}
\label{eq:avghadamard}
d(Tx,Ty)^2+(1-2\alpha)d(x,y)^2\leqslant2(1-\alpha)\Delta_T(x,y),\qquad\forall x\in D,\forall y\in E
\end{equation}
If $D=E$ then $T$ is $\alpha$-firmly nonexpansive on $D$. 
If $D=E=X$ the mapping 
$T$ is simply said to be $\alpha$-firmly nonexpansive. If the set of fixed points $\Fix T\neq\emptyset$ and $D=\Fix T, E=X$ then $T$ is said to be a quasi $\alpha$-firmly nonexpansive operator.   
\end{mydef}
Following this definition an operator $T:X\to X$ is said to be firmly nonexpansive if and only if it is $\alpha$-firmly nonexpansive with $\alpha=1/2$ i.e. $d(Tx,Ty)^2\leqslant \Delta_T(x,y)$ for all $x,y\in X$. \footnote{It was brought later to the attention of the author that operators satisfying such an inequality, surely not expressed in terms of the discrepancy $\Delta_T$, are said to satisfy property ($P_2$), see Lopez et al \cite{LopezP2}.} This notion of firm nonexpansiveness allows for a natural definition for what it might mean for an operator to be {\em monotone}. We can say that an operator $T$ is monotone whenever $\Delta_T(x,y)\geqslant 0$ for all $x,y\in X$. Then it is immediate that a firmly nonexpansive operator is monotone. This fact is in analogy with the same implication in the setting of a Hilbert space. However here we will not explore monotone operators. We leave them for future research.  
In the current literature as mentioned in the introduction, there exists a notion of firmly nonexpansiveness. Given an operator $T:X\to X$ and $x,y\in X$ let $x_t:=(1-t)x\oplus tTx$ and $y_t:=(1-t)y\oplus tTy$. Define the function $\phi_T(t):=d(x_t,y_t)$. The operator $T$ is said to be firmly nonexpansive if $\phi_T(t)$ is nonincreasing on $[0,1]$. It can be shown that this notion of firmly nonexpansiveness indeed implies our definition of firmly nonexpansiveness \cite[Proposition 7.5]{Berdellima}.
Note that when an operator $T$ is $\alpha$-firmly nonexpansive with $\alpha\leqslant 1/2$ then $\Delta_T(x,y)\geqslant 0$ for all $x,y\in X$. 
A metric space $(X,d)$ satisfies the Cauchy--Schwartz inequality if $|\langle\overrightarrow{xz},\overrightarrow{yw}\rangle|\leqslant d(x,z)d(y,w)$ for all $x,y,z,w\in X$. An important theorem of Berg and Nikolaev \cite[Theorem 1, Corollary 3]{Berg} states that a metric space satisfies Cauchy--Schwartz inequality if and 
only if it is a $\CAT(0)$ space. This has an important implication for our definition of firmly nonexpansiveness. An operator is nonexpansive whenever it is firmly nonexpansive. Therefore there is a degree of consistency of our notion of firmly nonexpansiveness. Afterall it is desirable that firmly nonexpansiveness be a stronger {\em contraction} property than nonexpansiveness. 

\subsection{Weak convergence} 
\label{ss:weak convergence}
\begin{mydef}
 \label{d:weakconvergence}
 Let $\Gamma_x(X)$ denote the set of all geodesic segments $\gamma:[0,1]\to X$ that emanate from $x$. A sequence $(x_n)_{n\in\mathbb{N}}\subseteq X$ {\em converges weakly} to $x$ and we denote it by $x_n\overset{w}\to x$ if and only if $\lim_nd(x,P_{\gamma}x_n)=0$ for all $\gamma\in \Gamma_x(X)$. 
\end{mydef}
Weak limits are unique. Indeed let $x,y$ be both weak limit points of a sequence $(x_n)_{n\in\mathbb{N}}$. Let $\gamma,\widetilde{\gamma}:[0,1]\to X$ be two geodesics such that $\gamma(0)=\widetilde{\gamma}(1)=x$ and $\gamma(1)=\widetilde{\gamma}(0)=y$. Because a $\CAT(0)$ space is uniquely geodesic then the image of both mappings $\gamma,\widetilde{\gamma}$ coincide i.e. $\gamma([0,1])=\widetilde{\gamma}([0,1])=[x,y]$. In particular $P_{\gamma}z=P_{\widetilde{\gamma}}z$ for all $z\in X$.
From triangle inequality we then obtain 
$$0\leqslant d(x,y)\leqslant d(x,P_{\gamma}x_n)+d(P_{\gamma}x_n,y)=d(x,P_{\gamma}x_n)+d(P_{\widetilde{\gamma}}x_n,y)\to 0\;\text{as}\;n\uparrow+\infty,$$
implying $x=y$. 
Moreover in a locally compact space weak and strong convergence coincide \cite[Theorem 3.25]{Berdellima}.
For more details refer to \cite[Chapter 3]{Berdellima}. 
An element $x\in X$ is an {\em asymptotic center} for a  sequence $(x_n)_{n\in\mathbb{N}}\subseteq X$ whenever 
\begin{equation}
 \label{eq:asymptoticcenter}
 \limsup_nd(x_n,x)\leqslant \limsup_nd(x_n,y),\;\forall y\in X.
\end{equation}
The asymptotic center of each sequence exists and it is unique because it can be realized as the minimizer of the strongly convex function $f(x;(x_n)_{n\in\mathbb{N}}):=\limsup_nd(x,x_n)^2$.
\begin{mydef}[$\Delta$-convergence, Lim \cite{Lim} (1976)]
\label{d:deltaconvergence}
A sequence $(x_n)_{n\in\mathbb{N}}$ $\Delta$-converges to $x$, $x_n\overset{\Delta}\to x$, whenever $x$ is the asymptotic center of each subsequence of $(x_n)_{n\in\mathbb{N}}$. 
\end{mydef}
On bounded sets weak convergence coincides with $\Delta$-convergence (Ba\v cak \cite[Proposition 3.1.3]{Bacak}). 
Therefore in a $\CAT(0)$ space $\Delta$-limits are unique. $\Delta$-convergence for metric spaces was later brought in the context of $\CAT(0)$ spaces by Kirk and Panyanak \cite{Kirk}.

\section{Calculus of quasi $\alpha$-firmly nonexpansive operators}
\label{s:calculus}

\subsection{Some preliminary results}
\begin{lemma}\label{l:afne implies ne}
Let $(X,d)$ be a $\CAT(0)$ space. Then
\begin{enumerate}[(i)]
\item \label{eq:one} An operator $T:X\to X$ is nonexpansive whenever $T$ is $\alpha$-firmly nonexpansive.
 \item  \label{eq:two} $T$ is quasi $\alpha$-firmly nonexpansive if and only if
\begin{equation}
 \label{eq:avghadamard2}
 d(Tx,y)^2\leqslant d(x,y)^2-\frac{1-\alpha}{\alpha}d(x,Tx)^2,\quad \forall x\in X.
\end{equation}
In particular if $T$ is quasi $\alpha$-firmly nonexpansive then it is quasi nonexpansive.
\end{enumerate}
\end{lemma}
\begin{proof}
 Note that \eqref{eq:avghadamard} can be equivalently written as 
\begin{equation}
 \label{eq:avghadamard1}
 d(Tx,Ty)^2\leqslant d(x,y)^2-\frac{1-\alpha}{\alpha}\Big(d(x,y)^2-2\Delta_T(x,y)+d(Tx,Ty)^2\Big).
\end{equation}
By Cauchy--Schwartz inequality it follows that
$$d(x,y)^2-2\Delta_T(x,y)+d(Tx,Ty)^2\geqslant(d(x,y)-d(Tx,Ty))^2\geqslant0.$$
This shows \eqref{eq:one}. Now let $y\in\Fix T$ then \eqref{eq:avghadamard1} reduces to \eqref{eq:avghadamard2}. Quasi nonexpansiveness follows immediately from \eqref{eq:avghadamard2}. This proves \eqref{eq:two}.
\end{proof}

\begin{lemma}\label{l:intersections}
Let $(X,d)$ be a metric space and let $S,T:X\to X$ be two operators such that $\Fix T\cap\Fix S\neq\emptyset$.
If $S$ is $\alpha$-firmly nonexpansive and 
$T$ is nonexpansive on 
 $\Fix TS\times(\Fix T\cap \Fix S)$.  Then 
 $\Fix TS=\Fix T\cap \Fix S$.
\end{lemma}
\begin{proof}
 The inclusion $\Fix T\cap \Fix S\subseteq \Fix TS$ is obvious. Now let 
 $x\in \Fix TS$ and $y\in \Fix T\cap \Fix S$. There are three mutually exclusive 
 cases. First let $Sx\in\Fix T$ then $Sx=TSx=x$ implies $x\in \Fix T\cap \Fix S$. 
 Second let $x\in \Fix S$ then $x=TSx=Tx$ implies $x\in \Fix T\cap \Fix S$. Finally, 
 let $x\notin \Fix S$ and $Sx\notin \Fix T$. This yields
\begin{align*}
d(x,y)^2=d(TSx,Ty)^2\leqslant d(Sx,y)^2&=d(Sx,Sy)^2\\
&\leqslant d(x,y)^2-\frac{1-\alpha}{\alpha}\Big(d(x,y)^2-2\Delta_S(x,y)+
d(Sx,Sy)^2\Big)
\end{align*}
where the first inequality follows from nonexpansiveness of $T$, and the 
second inequality follows from $\alpha$-firmly nonexpansiveness of $S$ with some constant $\alpha$ on $\Fix TS$.  
Assumption $y\in \Fix S$ and $x\notin \Fix S$ imply 
$$d(Sx,Sy)^2\leqslant d(x,y)^2-\frac{1-\alpha}{\alpha}d(x,Sx)^2<d(x,y)^2,$$
but $d(x,y)^2<d(x,y)^2$ is impossible. 
Therefore $\Fix TS\subseteq \Fix T\cap\Fix S$.
\end{proof}

The next result requires the following 
quantities:
\begin{subequations}\label{eq:LMUV}
\begin{eqnarray}
 && L(x,y):= d(x,y)^2-2\Delta_S(x,y)+d(Sx,Sy)^2;\label{eq:L}\\
 &&M(x,y):=d(Sx,Sy)^2-2\Delta_T(Sx,Sy)+d(TSx,TSy)^2;\label{eq:M}\\
 &&U(x,y):=\Delta_{TS}(x,y)+d(Sx,Sy)^2-\Delta_S(x,y)-\Delta_T(Sx,Sy);
 \label{eq:U}\\
 &&V(x,y):= d(x,y)^2-2\Delta_{TS}(x,y)+d(TSx,TSy)^2.\label{eq:V}
\end{eqnarray}
\end{subequations}

\begin{lemma}\label{t:afne of compositions}
 Let $(X,d)$ be a $\CAT(0)$ space.
 Let 
 $S:X\to X$ be $\alpha$-firmly nonexpansive with constant $\alpha_S$
 and let $T:X\to X$ be $\alpha$-firmly nonexpansive with constant 
 $\alpha_T$.  Then the composition $TS$ is $\alpha$-firmly nonexpansive with constant 
 \begin{equation}\label{eq:alpha comp}
 \alpha_{TS}\equiv\frac{\alpha_S+\alpha_T-2\alpha_S\alpha_T}{1-\alpha_S\alpha_T}
 \end{equation}
 whenever 
 \begin{equation}
\label{eq:avgineqH}
 \Big(\frac{1-\alpha_{S}}{\tau\alpha_{S}}\Big)^2L(x,y)+\Big(\frac{1-\alpha_{T}}{\tau\alpha_{T}}\Big)^2M(x,y)+
 2\Big(\frac{1-\alpha_{S}}{\tau\alpha_{S}}\Big)\Big(\frac{1-\alpha_{T}}{\tau\alpha_{T}}\Big)U(x,y)\geqslant 0,
 \qquad\forall x,y\in X,
\end{equation}
where 
 \begin{equation}\label{eq:tau}
 \tau\equiv \frac{1-\alpha_S}{\alpha_S}+\frac{1-\alpha_T}{\alpha_T}.
 \end{equation}
\end{lemma}
\begin{proof}
Since $T$ is $\alpha$-firmly nonexpansive with 
constant $\alpha_T$  we have 
$$d(TSx,TSy)^2+(1-2\alpha_T)d(Sx,Sy)^2\leqslant 2(1-\alpha_T)\Delta_T(Sx,Sy)
\qquad\forall x,y\in X.$$
By \eqref{eq:avghadamard1} this is equivalent to 
$$d(TSx,TSy)^2\leqslant d(Sx,Sy)^2-\frac{1-\alpha_T}{\alpha_T}\Big(d(Sx,Sy)^2-2\Delta_T(Sx,Sy)+d(TSx,TSy)^2\Big)
\quad\forall x,y\in X.$$
Since $S$ is $\alpha$-firmly nonexpansive with 
constant $\alpha_S$ we have 
\begin{align*}
d(TSx,TSy)^2\leqslant d(x,y)^2&
-\frac{1-\alpha_S}{\alpha_S}\Big(d(x,y)^2-2\Delta_S(x,y)+d(Sx,Sy)^2)\Big)\\
&-\frac{1-\alpha_T}{\alpha_T}\Big(d(Sx,Sy)^2-2\Delta_T(Sx,Sy)+d(TSx,TSy)^2\Big)
\end{align*}
for all $x,y\in X$.  
A short calculation yields 
\begin{eqnarray*}
&&\frac{1-\alpha_S}{\tau\alpha_S}L+\frac{1-\alpha_T}{\tau\alpha_T}M=
\Big(\frac{1-\alpha_S}{\tau\alpha_S}\Big)^2L+
\Big(\frac{1-\alpha_T}{\tau\alpha_T}\Big)^2M+\\
&&\qquad\qquad\qquad\qquad\qquad 2\Big(\frac{1-\alpha_S}{\tau\alpha_S}\Big)
    \Big(\frac{1-\alpha_T}{\tau\alpha_T}\Big)U+ \Big(\frac{1-\alpha_S}{\tau\alpha_S}\Big)
    \Big(\frac{1-\alpha_T}{\tau\alpha_T}\Big)V, 
\end{eqnarray*}
where $\tau$ is given by \eqref{eq:tau} and $L, M, U$ and $V$ are given by \eqref{eq:LMUV}.  
By the Cauchy--Schwarz inequality we have $L,M\geqslant 0$.  If 
inequality \eqref{eq:avgineqH} holds, then 
$$\frac{1-\alpha_S}{\tau\alpha_S}L+\frac{1-\alpha_T}{\tau\alpha_T}M
\geqslant \Big(\frac{1-\alpha_S}{\tau\alpha_S}\Big)
\Big(\frac{1-\alpha_T}{\tau\alpha_T}\Big)V.$$
Therefore 
$$d(TSx,TSy)^2\leqslant d(x,y)^2-\frac{1-\alpha_S}{\alpha_S}L-
\frac{1-\alpha_T}{\alpha_T}M\leqslant d(x,y)^2-
\frac{1}{\tau}\Big(\frac{1-\alpha_S}{\alpha_S}\Big)
\Big(\frac{1-\alpha_T}{\alpha_T}\Big)V.$$
Subsituting for $V, \tau$ and $\alpha_{TS}$ we obtain 
$$d(TSx,TSy)^2\leqslant d(x,y)^2-
\frac{1-\alpha_{TS}}{\alpha_{TS}}\Big(d(x,y)^2-2\Delta_{TS}(x,y)+d(TSx,TSy)^2\Big).$$
Rearranging terms the last inequality is equivalent to
$$d(TSx,TSy)^2+(1-2\alpha_{TS})d(x,y)^2\leqslant
2(1-\alpha_{TS})\Delta_{TS}(x,y)\quad\forall x,y\in X,$$
as claimed. 
\end{proof}
\begin{proposition}
\label{p:fixT}
 Let $(X,d)$ be a $\CAT(0)$ space and $T:X\to X$ be quasi nonexpansive. Then $\Fix T$ is a closed convex set.
\end{proposition}
\begin{proof}
  Clearly $\Fix T$ is nonempty by definition of quasi nonexpansiveness. Let $(x_n)_{n\in\mathbb{N}}\subseteq \Fix T$ be a sequence of fixed points of $T$ converging to some element $x\in X$. Then by triangle inequality and quasi nonexpansiveness we obtain
 $$d(Tx,x)\leqslant d(Tx,x_n)+d(x_n,x)\leqslant 2d(x_n,x).$$
 In the limit as $n\uparrow +\infty$ we obtain
 $d(Tx,x)=0$ and thus $x\in \Fix T$. Therefore $\Fix T$ is a closed set. Now let $x_1,x_2\in\Fix T$ and $s\in[0,1]$. Consider the element $x_s:=(1-s)x_1\oplus sx_2$. Then by strong convexity of the square of the metric with parameter $\mu=2$ we get
 $$d(Tx_s,x_s)^2\leqslant (1-s)d(Tx_s,x_1)^2+sd(Tx_s,x_2)^2-s(1-s)d(x_1,x_2)^2.$$
 Because $T$ is quasi nonexpansive then 
 $d(Tx_s,x_1)\leqslant d(x_s,x_1)$ and $d(Tx_s,x_2)\leqslant d(x_s,x_2)$. Applying once strong convexity we finally get the upper estimate
 \begin{align*}
  d(Tx_s,x_s)^2&\leqslant (1-s)[sd(x_1,x_2)^2-s(1-s)d(x_1,x_2)^2]\\&+s[(1-s)d(x_1,x_2)^2-s(1-s)d(x_1,x_2)^2]-s(1-s)d(x_1,x_2)^2=0.
 \end{align*}
Therefore $x_s\in \Fix T$. Since $x_1,x_2\in\Fix T, s\in[0,1]$ are arbitrary then $\Fix T$ is convex.  
\end{proof}

\begin{proposition}
\label{p:projectionfne}
 Let $(X,d)$ be a complete $\CAT(0)$ space and $C\subseteq X$ be a closed convex set. Then $P_C:X\to C$ is firmly nonexpansive, in particular it is nonexpansive.
\end{proposition}
\begin{proof}
 Let $x,y\in X$ be arbitrary. Denote by $P_Cx, P_Cy$ the metric projections onto $C$ of $x$ and $y$ respectively. They exist and are unique. Moreover utilizing inequality \eqref{eq:projectionsineq} we obtain
 \begin{align*}
  & d(x,P_Cx)^2+d(P_Cx,P_Cy)^2\leqslant d(x,P_Cy)^2\\
  & d(y,P_Cy)^2+d(P_Cy,P_Cx)^2\leqslant d(y,P_Cx)^2.
 \end{align*}
Adding both sides of the last two inequalities yields $d(P_Cx,P_Cy)^2\leqslant \Delta_{P_C}(x,y)$ and thus $P_C$ is firmly nonexpansive. By Lemma \ref{l:afne implies ne} \eqref{eq:one} $P_C$ is nonexpansive. This completes the proof. 
\end{proof}

\subsection{Compositions of quasi $\alpha$-firmly nonexpansive operators}
\begin{thm}[Compositions]\label{compositionsthm}
 Let $S$ be quasi $\alpha$-firmly nonexpansive  
 with constant $\alpha_S$, 
 let $T$ be quasi $\alpha$-firmly nonexpansive  
 with constant $\alpha_T$,  and let  
 $\Fix T\cap \Fix S\neq\emptyset$.   Then the operator 
 $TS$ is quasi $\alpha$-firmly nonexpansive  with constant $\alpha_{TS}$ given by \eqref{eq:alpha comp}.
\end{thm}
\begin{proof}
 By Lemma \ref{l:intersections}, it suffices to show that 
 inequality \eqref{eq:avgineqH} holds at all points  $y\in \Fix TS$. Note assumption $\Fix T\cap\Fix S\neq\emptyset$ implies by Lemma \ref{l:intersections} that $\Fix TS=\Fix T\cap\Fix S$. Then for $y\in \Fix TS$, we have  
 $L(x,y)=d(x,Sx)^2, M(x,y)=d(Sx,TSx)^2$ and $2U(x,y)=d(x,Sx)^2+d(Sx,TSx)^2-d(x,TSx)^2$, 
 where $L, M$ and $U$ are defined in \eqref{eq:LMUV}. 
Then  from \eqref{eq:avgineqH} it suffices to show that
 \begin{align*}
  \Big(\frac{1-\alpha_{S}}{\tau\alpha_{S}}\Big)^2d(x,Sx)^2&
  +\Big(\frac{1-\alpha_{T}}{\tau\alpha_{T}}\Big)^2d(Sx,TSx)^2\\
  &+\Big(\frac{1-\alpha_{S}}{\tau\alpha_{S}}\Big)\Big(\frac{1-\alpha_{T}}{\tau\alpha_{T}}\Big)\Big(d(x,Sx)^2
  +d(Sx,TSx)^2-d(x,TSx)^2\Big)\geqslant 0
 \end{align*}
for all $x\in H$, where $\tau$ is given by \eqref{eq:tau}. If we let
$\kappa:=\frac{1-\alpha_{S}}{\alpha_{S}}/\frac{1-\alpha_{T}}{\alpha_{T}}$
then it is equivalent to prove that
$$(\kappa+1)d(x,Sx)^2+\frac{\kappa+1}{\kappa}d(Sx,TSx)^2-d(x,TSx)^2\geqslant 0$$
for all $x\in H$. On the other hand we have the elementary inequality
$$\kappa d(x,Sx)^2+\frac{1}{\kappa}d(Sx,TSx)^2\geqslant 2d(x,Sx)d(Sx,TSx),
\qquad\forall\kappa>0$$
which together with the triangle inequality $d(x,Sx)+d(Sx,TSx)\geqslant d(x,TSx)$
imply
\begin{align*}
 (\kappa+1)d(x,Sx)^2+\frac{\kappa+1}{\kappa}d(Sx,TSx)^2
 \geqslant (d(x,Sx)+d(Sx,TSx))^2\geqslant d(x,TSx)^2, 
\end{align*}
for all $x\in X$, which completes the proof.  
\end{proof}

\begin{corollary}
 Let $(X,d)$ be a $\CAT(0)$ space and $(T_i)_{i=1}^N$ be a finite family of quasi $\alpha$-firmly nonexpansive operators with constants $\alpha_i\in(0,1)$. Then the operator $T:=T_{i_n}T_{i_{n-1}}...T_{i_1}$ where $i_j\in\{1,2,...,n\}$ are distinct is also quasi $\alpha$-firmly nonexpansive with some constant $\alpha\in(0,1)$. Moreover $\alpha$ is a function of $\alpha_i$-s and is determined recursively by the rule \eqref{eq:alpha comp}. 
\end{corollary}

\subsection{Convex combinations}
\subsubsection{Convex combinations of elements}
Let $(X,d)$ be a $\CAT(0)$ space. Convex combination for two elements $x,y\in X$ is well defined in terms of the geodesic connecting $x$ with $y$. The lack in general of an additive structure in $X$ makes the concept of \textit{convex combination} of more than two elements somewhat ambiguous. One way to define convex combinations of more than two elements would be the following. To keep arguments simple say we are given three points $x,y,z\in X$ and three numbers $w_1,w_2,w_3\in[0,1]$ such that $w_1+w_2+w_3=1$. Then the expressions
\begin{align*}
 &(w_1+w_2)\Big(\frac{w_1}{w_1+w_2}x\oplus\frac{w_2}{w_1+w_2}y\Big)\oplus w_3z\\&
 (w_2+w_3)\Big(\frac{w_2}{w_2+w_3}y\oplus\frac{w_3}{w_2+w_3}z\Big)\oplus w_1x\\&
 (w_1+w_3)\Big(\frac{w_1}{w_1+w_3}x\oplus\frac{w_3}{w_1+w_3}z\Big)\oplus w_2y
\end{align*}
seem all reasonable choices for a convex combination of $x,y,z$. But they are not guaranteed to be equal unless $X$. Notice that by construction all three expressions are in the convex hull $\co\{x,y,z\}$. Now which one to choose is not obvious. However there is way to define convex combinations uniquely based on the \textit{barycenter method}. To be more precise let $x_1,x_2,...,x_n\in X$ and $w_1,w_2,...,w_n\in[0,1]$ such that $\sum_iw_i=1$. The convex combination $x^*$ is the solution to the minimization problem
\begin{equation}
 \label{eq:cvxcombo}
 \min_{x\in X}F(x)=\sum_{i=1}^nw_id(x,x_i)^2.
\end{equation}
The existence and uniqueness of $x^*$ follows from \eqref{eq:cvxcombo} being strongly convex. We denote 
\begin{equation}
 \label{eq:cvxcombo1}
 x^*:=w_1x_1\oplus w_2x_2\oplus...\oplus w_nx_n.
\end{equation}

\subsubsection{Convex combinations of operators}
Having defined convex combinations of an arbitrary finite set of elements in $X$ then it is easy to define the convex combinations of operators. Let $T_1,T_2,...,T_n:X\to X$ be a family of operators. An operator $T:X\to X$ is said to be a \textit{convex combination} of $T_1,T_2,...,T_n$ for a given set of weights $w_1,w_2,...,w_n\in[0,1], \sum_iw_i=1$ if 
\begin{equation}
 \label{eq:cvxcombooperators}
 Tx=\arg\min_{y\in X}\sum_{i=1}^nw_id(y,T_ix)^2
\end{equation}
and we denote it by 
\begin{equation}
 \label{eq:cvxcombooperators1}
 T:=w_1T_1\oplus w_2T_2\oplus...\oplus w_nT_n.
\end{equation}

\begin{lemma}
\label{keylemmacvx}
 Let $T_i$ be quasi $\alpha$-firmly nonexpansive with $\alpha_i\in(0,1)$ for $i=1,2,...,n$ and $T$ be given by \eqref{eq:cvxcombooperators1} then $\bigcap_i\Fix T_i=\Fix T$ whenever $\bigcap_i\Fix T_i\neq\emptyset$.
\end{lemma}
\begin{proof}
 Assume $\bigcap_i\Fix T_i\neq\emptyset$ and let $x\in\bigcap_i\Fix T_i$ then by \eqref{eq:cvxcombooperators} follows
 $$Tx=\arg\min_{y\in X}\sum_{i=1}^nw_id(y,T_ix)^2=\arg\min_{y\in H}\sum_{i=1}^nw_id(y,x)^2=\arg\min_{y\in H}d(y,x)^2=x.$$
This shows $\bigcap_i\Fix T_i\subseteq\Fix T$. Now let $x\in \Fix T$ and $y\in \bigcap_i\Fix T_i$. Note that $d(\cdot,T_ix)^2$ is strongly convex with parameter $\mu=2$ for every $i=1,2,...,n$. Therefore the functional $F(\cdot):=\sum_iw_id(\cdot,T_ix)^2$ is strongly convex as a finite sum of strongly convex functions with parameter $\mu=\sum_i2w_i=2$. By virtue of Proposition \ref{prop1} the following inequality holds
\begin{equation}
 \label{eq:minimumineq}
 F(Tx)+d(Tx,y)^2\leqslant F(y),\;\forall y\in X.
\end{equation}
In particular since $x\in\Fix T$ then 
\begin{equation}
 \label{eq:minimumineq1}
 F(x)+d(x,y)^2\leqslant F(y),\;\forall y\in X.
\end{equation}
By assumption $T_i$ is quasi $\alpha$-firmly nonexpansive for all $i=1,2,...,n$. Then Lemma \ref{l:afne implies ne} \eqref{eq:one} implies
$$d(T_ix,y)^2\leqslant d(x,y)^2-\frac{1-\alpha_i}{\alpha_i}d(x,T_ix)^2,\;\forall x\in X, \forall y\in\Fix T_i.$$
Therefore in aggregate we obtain
$$\sum_{i=1}^nw_id(T_ix,y)^2\leqslant d(x,y)^2-\sum_{i=1}^nw_i\frac{1-\alpha_i}{\alpha_i}d(x,T_ix)^2,\;\forall x\in X,\forall y\in\bigcap_i\Fix T_i.$$
By inequality \eqref{eq:minimumineq1} and definition of $F$ we get
$$0\leqslant F(x)\leqslant-\sum_{i=1}^nw_i\frac{1-\alpha_i}{\alpha_i}d(x,T_ix)^2\leqslant 0,\;\forall x\in \Fix T.$$
Therefore $T_ix=x$ for all $i$ implies $\Fix T\subseteq\bigcap_i\Fix T_i$.  
\end{proof}

\begin{remark}
 Note that the above lemma still holds under slightly milder condition that the operator $T_i$ be \textit{strictly quasi nonexpansive} for every $i=1,2,...,n$, i.e.
 \begin{equation}
  \label{eq:quasine}
  d(T_ix,y)<d(x,y),\hspace{0.2cm}\forall y\in\Fix T_i, \forall x\in X\setminus\Fix T_i.
 \end{equation}
 Similarly for the operator $S$ in Lemma \ref{l:intersections} the condition can be relaxed to just strictly quasi nonexpansive.
\end{remark}

Suppose that $T_i$ is quasi $\alpha$-firmly nonexpansive with some constant $\alpha_i\in(0,1)$ for $i=1,2,...,n$. We would like to know under what conditions is $T=w_1T_1\oplus...\oplus w_nT_n$ a quasi $\alpha$-firmly nonexpansive operator. 

\begin{thm}[Convex combinations]
 \label{cxvcombooperator}
 Let $(X,d)$ be a $\CAT(0)$ space and $T_i:X\to X$ be a family of quasi $\alpha$-firmly nonexpansive operators with parameters $\alpha_i$ for $i=1,2,...,n$. If $\bigcap_i\Fix T_i\neq\emptyset$ then the operator $T:T\to T$ defined as in \eqref{eq:cvxcombooperators1} is quasi $\alpha$-firmly nonexpansive with parameter $\alpha=\max_i\{\alpha_i\}$.
\end{thm}
\begin{proof}
  Inequality \eqref{eq:minimumineq} implies $d(Tx,y)^2\leqslant F(y)$ for all $y\in H$ . By definition of the functional $F$  and assumption that $T_i$ is quasi $\alpha$-firmly nonexpansive for all $i=1,2,...,n$ by using Lemma  \ref{l:afne implies ne} \eqref{eq:one} we get  
 \begin{equation}
  \label{eq:Txineq}
  d(Tx,y)^2\leqslant d(x,y)^2-\sum_{i=1}^nw_i\frac{1-\alpha_i}{\alpha_i}d(T_ix,x)^2,\;\forall y\in\bigcap_i\Fix T_i.
 \end{equation}
This inequality is meaningful since by assumption $\bigcap_i\Fix T_i\neq\emptyset$. Moreover by Lemma \ref{keylemmacvx} we have $\Fix T=\bigcap_i\Fix T_i$. Therefore \eqref{eq:Txineq} is equivalent to
 \begin{equation}
  \label{eq:Txineq1}
  d(Tx,y)^2\leqslant d(x,y)^2-\sum_{i=1}^nw_i\frac{1-\alpha_i}{\alpha_i}d(T_ix,x)^2,\;\forall y\in\Fix T
 \end{equation}
 From \eqref{eq:Txineq1} we also get 
 \begin{equation}
  \label{eq:Txineq2}
  d(Tx,y)^2\leqslant d(x,y)^2-\frac{1-\alpha}{\alpha}\sum_{i=1}^nw_id(T_ix,x)^2,\;\forall y\in\Fix T
 \end{equation}
 where $\alpha:=\max_i\{\alpha_i\}$. By definition of $F$ this is the same as 
 \begin{equation}
  \label{eq:Txineq3}
  d(Tx,y)^2\leqslant d(x,y)^2-\frac{1-\alpha}{\alpha}F(x),\;\forall y\in\Fix T
 \end{equation}
 In \eqref{eq:minimumineq} we have in particular $d(Tx,x)^2\leqslant F(x)$. Therefore 
\begin{equation}
  \label{eq:Txineq4}
  d(Tx,y)^2\leqslant d(x,y)^2-\frac{1-\alpha}{\alpha}d(Tx,x)^2,\;\forall y\in\Fix T
 \end{equation} 
\end{proof}

%

\section{Cyclic and averaged projections}
\label{s:projection}

\subsection{A general theorem}
Let $(X,d)$ be a complete $\CAT(0)$ space and $T:X\to X$ be $\alpha$-firmly nonexpansive with some constant $\alpha\in(0,1)$. It is of interest to know of the convergence, both weak and strong, of the iterates $x_n:=Tx_{n-1}$ for $n\in\mathbb{N}$. In this subsection we establish a general convergence result.

Given a set $S\subseteq X$ and a sequence $(x_n)_{n\in\mathbb{N}}$ in $X$ we say $(x_n)_{n\in\mathbb{N}}$ is {\em Fej\'er monotone with respect to $S$} whenever
\begin{equation}
 d(x_n,y)\leqslant d(x_{n-1},y),\qquad \forall y\in S.
\end{equation}

\begin{lemma}
 \label{l:fejer}
 Let $T:X\to X$ be quasi $\alpha$-firmly nonexpansive. Then $x_n:=Tx_{n-1}$ for $n\in \mathbb{N}$ is Fej\'er monotone with respect to $\Fix T$. In particular $(x_n)_{n\in\mathbb{N}}$ is bounded.
\end{lemma}
\begin{proof}
By Lemma \ref{l:afne implies ne} \eqref{eq:two} the operator $T$ is quasi nonexpansive. Therefore for all $y\in\Fix T$ we obtain
$$d(x_{n},y)=d(Tx_{n-1},y)\leqslant d(x_{n-1},y),\;\forall n\in\mathbb{N}.$$
Thus $(x_n)_{n\in\mathbb{N}}$ is Fej\'er monotone with respect to $\Fix T$. An iterative application of the last inequality yields in particular $d(x_n,y)\leqslant d(x_0,y)$ for all $n\in\mathbb{N}$ and some $y\in \Fix T\subseteq X$. Hence $\sup_nd(x_n,y)\leqslant d(x_0,y)<+\infty$ implies
$(x_n)_{n\in\mathbb{N}}$ is bounded. 
\end{proof}

Given a sequence $(x_n)_{n\in\mathbb{N}}$ we say that $x\in X$ is a {\em weak cluster point} of $(x_n)_{n\in\mathbb{N}}$ if there exists some subsequence $(x_{n_k})_{k\in\mathbb{N}}$ of $(x_n)_{n\in\mathbb{N}}$ such that $x_{n_k}\overset{w}\to x$.

\begin{lemma}
\label{l:weakcluster}
 Let $(X,d)$ be a complete $\CAT(0)$ space and $C\subseteq X$ a closed convex set. Let $(x_n)_{n\in\mathbb{N}}$ be Fej\'er monotone with respect to $C$. Then $x_n\overset{w}\to x$ for some $x\in C$ if and only if all  weak cluster points of $(x_n)_{n\in\mathbb{N}}$ belong to $C$. In particular this weak limit and weak cluster points all coincide. 
\end{lemma}
\begin{proof}
 A Fej\'er monotone sequence is bounded by virtue of Lemma \ref{l:fejer}. Since weak and $\Delta$-convergence coincide on bounded sets \cite[Proposition 3.1.3]{Bacak} then the result follows immediately from \cite[Proposition 3.2.6 $(iii)$]{Bacak}.
\end{proof}

\begin{thm}
 \label{th:general}
 Let $(X,d)$ be a complete $\CAT(0)$ space and $T:X\to X$ a nonexpansive operator. If $T$ is quasi $\alpha$-firmly nonexpansive then the iterates $x_n:=Tx_{n-1}$ for $n\in\mathbb{N}$ converge weakly to some $x^*\in\Fix T$. Moreover $\lim_nd(P_{\Fix T}x_n,x^*)=0$ whenever
\begin{equation}
  \label{eq:technical2}
  \lim\sup_n(d(x_n,P_{\gamma}P_{\Fix T}x_n)-d(x_n,P_{\Fix T}x_n))=0,\qquad\forall\gamma\in\Gamma_{x^*}(\Fix T)
 \end{equation}
 where $\Gamma_{x^*}(\Fix T):=\{\gamma\in\Gamma_{x^*}(X)\hspace{0.1cm}|\hspace{0.1cm}\gamma\subset \Fix T\}$.
\end{thm}
\begin{proof}
 We follow the lines in the proof of \cite[Theorem 7.16]{Berdellima}. Let $T:X\to X$ be $\alpha$-firmly nonexpansive operator. By Lemma \ref{l:fejer} the iterates $x_n:=Tx_{n-1}$ for $n\in\mathbb{N}$ form a Fej\'er monotone sequence with respect to $\Fix T$ and in particular a bounded sequence. Hence there is a subsequence $(x_{n_k})_{k\in\mathbb{N}}$ of $(x_n)_{n\in\mathbb{N}}$ weakly converging to some $x^*\in X$. We claim that $x^*\in\Fix T$. Note that by quasi $\alpha$-firmly nonexpansiveness we have
 $$d(x_n,Tx_n)^2\leqslant \frac{\alpha}{1-\alpha}\Big(d(x_n,y)^2-d(Tx_n,y)^2\Big),\;\forall y\in\Fix T.$$
 Since $(x_n)_{n\in\mathbb{N}}$ is Fej\'er monotone then the sequence of nonnegative real numbers $(d(x_n,y))_{n\in\mathbb{N}}$ is bounded and monotone for every $y\in \Fix T$, therefore by Bolzano--Weierstrass theorem it converges to some real number $d(y)$ possibly depending on $y$. We then obtain 
 $$\lim_nd(x_n,Tx_n)^2\leqslant \frac{\alpha}{1-\alpha}\lim_n\Big(d(x_n,y)^2-d(Tx_n,y)^2\Big)= \frac{\alpha}{1-\alpha}\Big(d(y)^2-d(y)^2\Big)=0.$$
 On the other hand by triangle inequality and nonexpansiveness of $T$ we get
 $$\limsup_kd(Tx^*,x_{n_k})\leqslant \limsup_kd(Tx^*,Tx_{n_k})+\limsup_kd(Tx_{n_k},x_{n_k})\leqslant \limsup_kd(x^*,x_{n_k}).$$
 Because $x_{n_k}\overset{w}\to x^*$ then $x^*$ is an asymptotic center for the subsequence $(x_{n_k})_{k\in\mathbb{N}}$. Last inequality implies that $Tx^*$ is also an asymptotic center for $(x_{n_k})_{k\in\mathbb{N}}$. By uniqueness of asymptotic centers then $Tx^*=x^*$. Therefore $x^*\in\Fix T$. Now if $(x_{n_m})_{m\in\mathbb{N}}$ is another subsequence converging weakly to some $x^{**}\in X$ by same arguments we obtain $x^{**}\in\Fix T$. Thus all weak cluster points of $T$ are in $\Fix T$. By Lemma \ref{p:fixT} the set $\Fix T$ is closed and convex which together with Fej\'er monotonicity of $(x_n)_{n\in\mathbb{N}}$ by virtue of Lemma \ref{l:weakcluster} imply $x^*=x^{**}$ and $x_n\overset{w}\to x^*\in\Fix T$.  
 
 Denote by $\bar{x}_n:=P_{\Fix T}x_n$ for all $n\in\mathbb{N}$. By \eqref{eq:projectionsineq} we have
 \begin{equation}
  \label{eq:Cauchy}  d(x_m,\bar{x}_m)^2+d(\bar{x}_m,\bar{x}_n)^2\leqslant d(\bar{x}_m,\bar{x}_n)^2
 \end{equation}
 which together with Fej\'er monotonicity implies
 \begin{equation}
  \label{eq:Cauchy}
 d(\bar{x}_m,\bar{x}_n)^2\leqslant d(x_m,\bar{x}_n)^2- d(x_m,\bar{x}_m)^2\leqslant d(x_n,\bar{x}_n)^2- d(x_m,\bar{x}_m)^2
 \end{equation}
 whenever $m\geqslant n$. Taking limit in \eqref{eq:Cauchy} as $m,n\to+\infty$ gives $\lim_{m,n}d(\bar{x}_m,\bar{x}_n)=0$ hence $(\bar{x}_n)_{n\in\mathbb{N}}$ is a Cauchy sequence in $\Fix T$. Because the set $\Fix T$ is closed and hence complete then $\bar{x}_n\to \bar{x}\in \Fix T$. Let $\gamma:[0,1]\to X$ be the geodesic segment joining $\bar{x}$ with $x^*$ such that $\gamma(0)=x^*,\gamma(1)=\bar{x}$. Note that $\gamma\in\Gamma_{x^*}(X)$ and $\gamma\subset \Fix T$ since $\Fix T$ is convex. Again by \eqref{eq:projectionsineq} we obtain the inequalities
 \begin{equation}
  \label{eq:chineq1}
  d(\bar{x}_n,P_{\gamma}x_n)^2+d(x_n,\bar{x}_n)^2\leqslant d(x_n,P_{\gamma}x_n)^2
 \end{equation}
 and 
 \begin{equation}
  \label{eq:chineq2}
  d(P_{\gamma}\bar{x}_n,P_{\gamma}x_n)^2+d(x_n,P_{\gamma}x_n)^2\leqslant d(x_n,P_{\gamma}\bar{x}_n)^2.
 \end{equation}
Adding \eqref{eq:chineq1} and \eqref{eq:chineq2} and passing in the limit we obtain
$$2d(\bar{x},x^*)^2\leqslant\lim\sup_n(d(x_n,P_{\gamma}\bar{x}_n)^2-d(x_n,\bar{x}_n)^2)=0.$$
Therefore $\bar{x}=x^*$. This completes the proof.
\end{proof}

\begin{corollary}
 Let $(X,d)$ be a complete $\CAT(0)$ space and $T:X\to X$ be $\alpha$-firmly nonexpansive operator. Then the iterates $x_n:=Tx_{n-1}$ for $n\in\mathbb{N}$ converge weakly to some $x^*\in\Fix T$. Moreover $\lim_nd(P_{\Fix T}x_n,x^*)=0$ whenever the condition \eqref{eq:technical2} is satisfied.
\end{corollary}

\begin{corollary}
 Let $(X,d)$ be a complete locally compact $\CAT(0)$ space and $T:X\to X$ a nonexpansive operator. If $T$ is quasi $\alpha$-firmly nonexpansive then the iterates $x_n:=Tx_{n-1}$ for $n\in\mathbb{N}$ converge strongly to some $x^*\in\Fix T$.
\end{corollary}

\subsection{Cyclic projections}
Let $(X,d)$ be a complete $\CAT(0)$ space $(C_i)_{i=1}^N$ be a family of closed convex sets in $X$. Denote by $P_i:=P_{C_i}$ the metric projection onto $C_i$ for all $i=1,2,...,N$. Suppose that $\bigcap_i C_i\neq\emptyset$ and let $C:=\bigcap_iC_i$. Denote by $P_C$ the metric projection onto $C$. For a given arbitrary point $x\in X$ the \textit{cyclic projection method} is defined as follows 
\begin{equation}
 \label{eq:cyclicproj}
 x_0:=x\quad\text{and}\quad x_n:=P_{[n]}x_{n-1}, n=1,2,3,...
\end{equation}
where $[n]:=n(\mod N)+1\in\{1,2,...,N\}$. In particular $x_{nN}=(P_NP_{N-1}...P_1)^nx_0$. 

\begin{proposition}
\label{p:P}
Let $P:=P_{N}P_{N-1}...P_1$. Then $\Fix P=\Fix P_C$.
\end{proposition}
\begin{proof}
By Proposition \ref{p:projectionfne} for every $i=1,2,...,N$ we have
$$d(P_ix,P_iy)^2\leqslant\Delta_{P_i}(x,y),\qquad\forall x,y\in X.$$
In particular 
$$d(P_ix,y)^2\leqslant\Delta_{P_i}(x,y), \forall x\in X, \forall y\in \Fix P_i.$$
Therefore $P_i$ is quasi firmly nonexpansive operator for every $i=1,2,...,n$. Note that $\Fix P=C$ and $\Fix P_i=C_i$ for every $i=1,2,...,n$.
By definition the operator $P$ is a composition of quasi firmly nonexpansive operators. An iterative application of Theorem \ref{compositionsthm} implies that $P$ is also quasi firmly nonexpansive and $$\Fix P=\bigcap_i\Fix P_i=\bigcap_i C_i=C=\Fix P_C.$$
\end{proof}
 

\begin{thm}
\label{th:cyclic}
Let $(X,d)$ be a complete $\CAT(0)$ space.
The sequence $(x_n)_{n\in\mathbb{N}}$ generated by \eqref{eq:cyclicproj} converges weakly to some element $x^*\in C$.  Moreover $\lim_nd(P_Cx_n,x^*)=0$ whenever 
 \begin{equation}
  \label{eq:technical}
  \lim\sup_n(d(x_n,P_{\gamma}P_Cx_n)-d(x_n,P_Cx_n))=0,\qquad\forall\gamma\in\Gamma_{x^*}(C)
 \end{equation}
 where $\Gamma_{x^*}(C):=\{\gamma\in\Gamma_{x^*}(H)\hspace{0.1cm}|\hspace{0.1cm}\gamma\subset C\}$.
\end{thm}
\begin{proof}
It suffices to show for the subsequence $(\hat{x}_{n})_{n\in\mathbb{N}}$ where $\hat{x}_{n}:=x_{nN}$ for each $n\in\mathbb{N}$. By definition $\hat{x}_{n}=P^nx_0$ where $P=P_{N}P_{N-1}...P_1$. 
By Proposition \ref{p:projectionfne} the operators $P_i$ are nonexpansive and therefore so is $P$. Moreover by Theorem \ref{compositionsthm} it follows that $P$ is quasi firmly nonexpansive. By Proposition \ref{p:P} we have $\Fix P=C$. An application of Theorem \ref{th:general} then yields the conclusion. This completes the proof.
\end{proof}

\begin{corollary}
 The sequence $(x_n)_{n\in\mathbb{N}}$ generated by \eqref{eq:cyclicproj} converges strongly to some element $x^*\in C$ whenever the underlying $\CAT(0)$ space $(X,d)$ is complete and locally compact.
\end{corollary}

\textbf{Historical remarks:} Cyclic projections algorithm has a long history in mathematics. The basic theorem for the case of two intersecting subspaces in a Hilbert space is due to von Neumann \cite{vNeumann}, thereafter it was generalized by Halpern \cite{Hcyclic} for an arbitrary finite number of intersecting subspaces. However this generalization was in fact discovered earlier in 1937 independently by S. Kaczmarz for solving a system of linear equations \cite{Kaczmarz}. This iterative method is also known as {\em Kaczmarzs algorithm} or {\em Kaczmarzs method}. 
The first statement about cyclic projections for a finite number of intersecting closed convex sets in a Hilbert space was proved by Bregman \cite{Bregman}. Baillon and Brezis \cite{Brezis} showed that $(P_C(x_n))_{n\in\mathbb{N}}$ converged in norm to some element of $C$. Bauschke \cite[ Theorem
 6.2.2 $(iii)$]{BauschkePhD} established in his PhD Thesis that the norm limit of $(P_C(x_n))_{n\in\mathbb{N}}$ was in fact the weak limit of $(x_n)_{n\in\mathbb{N}}$. 
In this historical context our Theorem \ref{th:cyclic} can be regarded as an extension of the classical result in Hilbert spaces to the broader class of complete $\CAT(0)$ spaces. However a distinction has to be made. Whereas the strong convergence of $(P_Cx_n)_{n\in\mathbb{N}}$ to the weak limit $x^*$ of $(x_n)_{n\in\mathbb{N}}$ is unconditional in a Hilbert space, in a complete $\CAT(0)$ space we add condition \eqref{eq:technical} necessity of which needs to be investigated.

\subsection{Averaged projections}
Let $(C_i)_{i=1}^N$ be a family of closed convex sets in a complete $\CAT(0)$ space with nonempty intersection $C\neq\emptyset$. Let $P_i$ denote the metric projection onto $C_i$ as before and $P_C$ the metric projection onto $C$. Let 
\begin{equation}
 \label{avgprojop}
 P:=\frac{1}{N}P_1\oplus\frac{1}{N}P_2\oplus...\oplus\frac{1}{N}P_N
\end{equation}
For a given element $x\in H$ the {\em averaged projections method} is defined as 
\begin{equation}
 \label{avgprojalg}
 x_0:=x\quad\text{and}\quad x_n:=Px_{n-1},n=1,2,3,...
\end{equation}
Note that by definition \eqref{avgprojop} the element $Px_{n-1}$ solves the problem 
\begin{equation}
 \label{avgprojmin}
 \min_{y\in X}\sum_{i=1}^N\frac{1}{N}d(y,P_ix_{n-1})^2,\qquad n=1,2,3,...
\end{equation}
For each $n\in\mathbb{N}$ the element $Px_{n-1}$ exists and it is unique. 

\begin{proposition}
 \label{avgprojprop}
 The operator $P$ defined by \eqref{avgprojop} is a mapping from $X$ onto $X$ that is quasi firmly nonexpansive. Moreover $\Fix P=C$.
\end{proposition}
\begin{proof}
 It is evident by definition that $P$ is a mapping from $X$ onto $X$. By Theorem \ref{cxvcombooperator} the operator $P$ is quasi firmly nonexpansive as a convex combination of quasi firmly nonexpansive operators with weights $w_i=1/N$ for all $i=1,2,...,N$. Since by assumptions $C\neq\emptyset$ and $P_i$ is quasi firmly nonexpansive for all $i=1,2,...,N$ then Lemma \ref{keylemmacvx} implies $$\Fix P=\bigcap_i\Fix P_i=\bigcap_i C_i=C.$$ This completes the proof.
\end{proof}


\begin{thm}
\label{th:avg}
Let $(X,d)$ be a complete $\CAT(0)$ space.
The sequence $(x_n)_{n\in\mathbb{N}}$ generated by \eqref{avgprojalg} converges weakly to some element $x^*\in C$.  Moreover $\lim_nd(P_Cx_n,x^*)=0$ whenever \eqref{eq:technical} is satisfied.
\end{thm}
\begin{proof}
 Similar to the proof of Theorem \ref{th:cyclic} but using Proposition \ref{avgprojprop} instead.
\end{proof}

\begin{corollary}
In a complete locally compact $\CAT(0)$ space
the sequence $(x_n)_{n\in\mathbb{N}}$ generated by \eqref{avgprojalg} converges strongly to some element $x^*\in C$.
\end{corollary}

Note that in literature \eqref{avgprojop}-\eqref{avgprojalg} is also known as the {\em Cimmino's method} (see for instance Bauschke and Borwein's review on this subject \cite{BauschkeBorwein} and references therein).

\bibliographystyle{plain}
\bibliography{literature}

%
%
%

\end{document}